\documentclass[11pt,a4paper,twoside,reqno]{amsart}
\usepackage[utf8]{inputenc}
\usepackage[margin=0.9in]{geometry} % Required for inputting international characters
\usepackage{amsthm}
\usepackage{amsmath}
\usepackage{amsfonts}
\usepackage{amssymb}
\usepackage{url}
\usepackage{hyperref}
\usepackage[makeroom]{cancel}
\usepackage{enumerate}
\usepackage{mathrsfs}
\usepackage{xcolor}
\usepackage{dsfont}
\usepackage{mathtools}
\mathtoolsset{showonlyrefs}

\newtheorem{thm}{Theorem}%[section]

\newtheorem{lem}{Lemma}[section]

\newcommand{\supp}{\operatorname{supp}}

\counterwithin{Theorem}{section}
\numberwithin{equation}{section}

\allowdisplaybreaks

\title{Counterexamples to the inhomogeneous Duffin--Schaeffer Conjecture}
\author{Manuel Hauke-Treuer, James Maynard, Andrew Pollington}
\begin{document}

\newcommand{\Addresses}{{
  \bigskip
  \footnotesize

  Manuel Hauke-Treuer, \textsc{Institute of Analysis and Number Theory, TU Graz},\par 
  \textit{E-mail address:} \texttt{hauke@math.tugraz.at}\vspace{5mm}

  James Maynard, \textsc{Mathematical Institute, University of Oxford},\par
   \textit{E-mail address:} \texttt{james.alexander.maynard@gmail.com}\vspace{5mm}
   
  Andrew Pollington\par
  \textit{E-mail address:} \texttt{adpoll@yahoo.com} 
  }
  }

	\begin{abstract}
		The Duffin--Schaeffer Conjecture, proven in 2020, gives a zero-full dichotomy for the approximation of irrational numbers with reduced fractions for arbitrary approximation functions. We construct approximation functions such that the inhomogeneous analogue of the Duffin--Schaeffer Conjecture fails.
		The counterexamples are obtained for all non-zero rationals and certain Liouville numbers.
	\end{abstract}
	\maketitle
\section{Introduction and statement of results}
    
	Khintchine's Theorem \cite{k24,k26}	gives a zero-full dichotomy on the set of $\psi$-approximable numbers:
	Writing 
	
	\[W(\psi) := \left\{\alpha \in [0,1): \left\lvert \alpha - \frac{a}{q}\right\rvert \leq \frac{\psi(q)}{q} \quad \text{for infinitely many } (a,q) \in \mathbb{Z} \times \mathbb{N}\right\},\]
	we have for every monotonic approximation function $\psi: \mathbb{N} \to [0,\infty)$ that
	
	\[\lambda(W(\psi)) = \begin{cases} 0 \quad &\text{ if } \sum_{q \in \mathbb{N}} \psi(q) < \infty, \\ 
		1 &\text{ if } \sum_{q \in \mathbb{N}} \psi(q) = \infty,
	\end{cases}\]
	where $\lambda$ denotes the Lebesgue measure.\\
	
	Duffin and Schaeffer \cite{DS41} showed that the monotonicity assumption in Khintchine's Theorem is necessary. Their counterexamples exploit the non-uniqueness of non-reduced representations of rational numbers, which led them to conjecture that the theorem becomes true when considering only reduced fractions and adapts the divergence condition accordingly. This was confirmed by Koukoulopoulos and the second-named author \cite{KM}:
	Writing  
	\[W'(\psi) := \left\{\alpha \in [0,1): \left\lvert \alpha - \frac{a}{q}\right\rvert \leq \frac{\psi(q)}{q} \quad \text{for infinitely many } (a,q) \in \mathbb{Z} \times \mathbb{N}, \gcd(a,q) = 1\right\},\]
	we have for arbitrary (not necessarily monotonic) $\psi: \mathbb{N} \to [0,\infty)$
	\[\lambda(W'(\psi)) = \begin{cases} 0 \quad&\text{ if } \sum_{q \in \mathbb{N}} \frac{\psi(q)\varphi(q)}{q} < \infty, \\ 
		1 &\text{ if } \sum_{q \in \mathbb{N}} \frac{\psi(q)\varphi(q)}{q} = \infty,\end{cases}\]
        where $\varphi$ denotes the Euler totient function.\\
        
	In the inhomogeneous setup, one fixes a shift $\gamma \in \mathbb{R}$ and asks for 
	
	\[\left\lvert \alpha - \frac{a + \gamma}{q}\right\rvert \leq \frac{\psi(q)}{q} \quad \text{for infinitely many } (a,q) \in \mathbb{Z} \times \mathbb{N}\]
	instead. The generalization of Khintchine's Theorem due to Szüsz \cite{szusz} shows that the set
	
	\[W(\psi,\gamma) := \left\{\alpha \in [0,1): \left\lvert \alpha - \frac{a + \gamma}{q}\right\rvert \leq \frac{\psi(q)}{q} \quad\text{for infinitely many } (a,q) \in \mathbb{Z} \times \mathbb{N}\right\}\]
	satisfies
	
	\[\lambda(W(\psi,\gamma)) = \begin{cases} 0 \textit{ if } \sum_{q \in \mathbb{N}} \psi(q) < \infty, \\ 
		1 \textit{ if } \sum_{q \in \mathbb{N}} \psi(q) = \infty,
	\end{cases}\]
	again subject to $\psi: \mathbb{N} \to [0,\infty)$ being monotonically decreasing\footnote{Khintchine and Szüsz assumed the stronger condition of $q \mapsto q\psi(q)$ to be decreasing; the weaker monotonicity assumption on $\psi$ is due to Schmidt \cite{S64}.}.\\
	
	Ram\'{i}rez \cite{R17} extended the Duffin--Schaeffer counterexamples to the inhomogeneous setup, showing that for every fixed $\gamma \in \mathbb{R}$, there exists a non-monotonic $\psi: \mathbb{N} \to [0,\infty)$ such that $\sum_{q \in \mathbb{N}} \psi(q) = \infty$, while $\lambda(W(\psi,\gamma)) = 0$. Motivated by the homogeneous conjecture, he suggested the direct analogue, which he called the "Inhomogeneous Duffin--Schaeffer Conjecture". Since then, the conjecture and related questions concerning non-monotonic inhomogeneous approximation have been studied in several works, see e.g. \cite{BHV24, CT24, HR24, Yu} and the references therein.\\
	
	We show that the conjecture fails for every non-zero rational $\gamma$, and extend the construction to irrationals that are sufficiently well approximable by rational numbers:
	
	\begin{thm}\label{thm_counter}
	Let $\gamma \neq 0$ satisfy 
	\begin{equation}\label{dioph_cond}\left\lvert \gamma - \frac{r}{s}\right\rvert < \exp(-s^{12}) \text{ for infinitely many } (r,s) \in \mathbb{Z} \times \mathbb{N}.\end{equation}
	Then writing
	\[W'(\psi,\gamma) := \left\{\alpha \in [0,1): \left\lvert \alpha - \frac{a + \gamma}{q}\right\rvert \leq \frac{\psi(q)}{q} \quad\text{for infinitely many } (a,q) \in \mathbb{Z} \times \mathbb{N}, \gcd(a,q) = 1\right\},\]
	 there exists $\psi: \mathbb{N} \to [0,\infty)$ with 
		$\sum_{q \in \mathbb{N}} \frac{\psi(q)\varphi(q)}{q} = \infty,$ but $\lambda(W'(\psi,\gamma)) = 0$.
	\end{thm}
	\vspace{1cm}
    
Some remarks concerning the above theorem are in order:
	
	\begin{enumerate}
		\item We do not expect the precise Diophantine assumption in \eqref{dioph_cond} to be sharp, and some optimization could potentially weaken it to some degree. However, it is clear that our approach is restricted to exceptionally well-approximable $\gamma$. The behaviour for arbitrary, in particular non-Liouville $\gamma$, remains widely open. It is not clear whether the Diophantine properties of the shift play a decisive role. In related inhomogeneous Khintchine-type problems, positive results of Yu \cite{Yu} hold under suitable Diophantine assumptions on the shift, including all non-Liouville irrationals, whereas due to \cite{CHPR}, the result breaks down for sufficiently well-approximable shifts.
		\item The role of the coprimality condition in the inhomogeneous setup no longer has the same effect of making fractions unique, which plays a crucial role in the counterexamples we construct --- see Section \ref{sec_heur} for a detailed discussion and a heuristic explanation of the underlying principle.
        Instead of classical coprimality, one may impose a reduction condition depending on $\gamma$, usually referred to as ``shift-reduced''. Such a result was already obtained by the first-named author with Beresnevich--Velani \cite[Theorem 14]{BHV24} in the rational setting: For $\gamma = \frac{A}{B} \in \mathbb{Q}, \gcd(A,B) = 1$, and defining the shift-reduced sets
        
		\[W_s'\left(\psi,\frac{A}{B}\right) := \left\{\alpha \in [0,1): \left\lvert \alpha - \frac{a + \frac{A}{B}}{q}\right\rvert \leq \frac{\psi(q)}{q} \quad \text{for i.m. } (a,q) \in \mathbb{Z} \times \mathbb{N}, \gcd(aB + A,q) = 1 \right\},\]
		we have for arbitrary $\psi: \mathbb{N} \to [0,\infty)$ that
		\[\lambda\left(W_s'\left(\psi,\frac{A}{B}\right)\right) = \begin{cases} 0 \quad&\text{ if } \sum_{q \in \mathbb{N}} \frac{\psi(q)\varphi(q)}{q} < \infty, \\ 
			1 &\text{ if } \sum_{q \in \mathbb{N}} \frac{\psi(q)\varphi(q)}{q} = \infty.\end{cases}\]
            Note that this recovers as special case the homogeneous case $\gamma = 0$ from \cite{KM}, and extends to all rational numbers. Again, the behaviour for irrational $\gamma$ remains open. We remark that the sets $W_s'(\psi,\frac{A}{B})$ are invariant modulo $1$ in the second argument. While this seems a natural condition, we note that Theorem \ref{thm_counter} also applies to integer shifts $\gamma \in \mathbb{Z} \setminus \{0\}$.
		\item The situation is different if one additionally randomizes the shift parameter, or considers simultaneous approximation or approximation by linear forms. In all these cases, the generalization of the Khintchine--Szüsz Theorem to the respective setups is known to hold without any monotonicity assumption \cite{ahr,ar,bv,cassels,h_quant,kim}.
		\item While finalizing this draft, we were informed by Lingmin Liao and Yubin He that they constructed independently, with the help of AI, counterexamples to the inhomogeneous Duffin--Schaeffer Conjecture. Their result also covers all non-zero rational shifts and some exceptional irrational shifts. The counterexample construction is essentially the same, while some details of the proofs differ.
	\end{enumerate}
	
	\subsection*{Acknowledgments}
	MHT was funded by the Austrian Science Fund (FWF) project\\ 10.55776/ESP5134624. JM is supported by the European Research Council (ERC) under the European Union’s Horizon programme (grant agreement No 101230920). MHT would like to thank Victor Beresnevich for discussions concerning potential counterexamples, and the University of Oxford for their hospitality while he was visiting JM. AP thanks the University of Oxford for their hospitality while visiting JM. The authors would like to thank Lingmin Liao and Yubin He for sharing a previous version of their article with us. 

\section{Heuristic overview of the proof}\label{sec_heur}

Let us first recall the original counterexample of Duffin and Schaeffer \cite{DS41} to a putative extension of Khintchine's Theorem without the monotonicity assumption. Their key idea is that since we can trivially have a single rational number represented as $a/q$ for many different $q$'s (since $1/2=2/4=3/6=\dots$), we can engineer significant overlaps in the set of numbers approximated by rationals with denominator $q$.

Let $\Pi:= \Pi_x := \prod_{x\le p\le e^x}p$ be a product of many consecutive primes of some scale $x$, and $\psi(q):= \varepsilon q/\Pi$ when $q|\Pi$, and $\psi(q)=0$ otherwise (here $\varepsilon>0$ is a fixed constant depending only on the scale $x$). Then it is easy to see that
\[
E_q:=\Bigl\{\alpha\in [0,1):\,\Bigl|\alpha-\frac{a}{q}\Bigr|\le \frac{\psi(q)}{q}\Bigr\}
\]
has measure $\lambda(E_q)=\varepsilon q/\Pi$ for each $q|\Pi$. Thus we see that
\[
\sum_{q}\lambda(E_q)=\frac{\varepsilon}{\Pi} \sum_{q|\Pi}q=\varepsilon \prod_{p|\Pi}\Bigl(1+\frac{1}{p}\Bigr).
\]
On the other hand, since every set $E_q$ is a collection of small intervals about rationals of the form $a/q$ with $q|\Pi$, we see that
\[
\bigcup_q E_q\subseteq \bigcup_{b\pmod{\Pi}}\Bigl[\frac{b}{\Pi}-\frac{\varepsilon}{\Pi},\frac{b}{\Pi}+\frac{\varepsilon}{\Pi}\Bigr],
\]
so $\lambda(\cup E_q)\le 2\varepsilon$. Thus the measure of the union $\lambda(\cup E_q)$ can be arbitrarily small compared with the sum $\sum_q \lambda(E_q)$ of the individual measures if $x$ is large enough, since $\prod_{p|\Pi}(1+1/p)\approx \log{2^x}/\log{x}$. By combining an infinite collection of different scales together, this leads to the counterexample. The key feature making this work was that a rational $b/\Pi$ could typically be written as $a/q$ with $q|\Pi$ in many different ways so there were significant overlaps between the $E_q$ for $q|\Pi$; the average number of ways was $\prod_{p|\Pi}(1+1/p)$, which could become arbitrarily large.

If, as Duffin and Schaeffer proposed, we make a restriction in $E_q$ to only considering reduced fractions $a/q$ then this example completely breaks down, since every rational $b/\Pi$ would then have a unique representation as $a/q$ with $(a,q)=1$.

Now let us consider the inhomogeneous situation, with $\gamma=1$ for simplicity. Keeping the same setup, we would be interested in the number of representations of $b/\Pi$ as $(a+1)/q$ with $q|\Pi$ and $(a,q)=1$. We have seen that the average number of representations of $b/\Pi$ as $(a+1)/q$ is $\prod_{p|\Pi}(1+1/p)$, and random heuristics might lead us to guess any such representative to satisfy $(a,q)=1$ with `probability' $\prod_{p|q}(1-1/p)$. A quick calculation reveals that on average we would then expect $b/\Pi$ to have exactly $1$ representative as $(a+1)/q$ with $(a,q)=1$; the restriction $(a,q)=1$ would hold with a probability that perfectly counteracts the many representations. This appears to mean that this counterexample does not apply to the homogeneous situation.

However, the exact number of representations $b/\Pi$ has of the form $(a+1)/q$ is the number of common divisors of $b$ and $\Pi$. Although on \textit{average} $b$ has $\prod_{p|\Pi}(1+1/p)$ such divisors, \textit{most} choices of $b$ have many fewer representations (namely about $2^{\sum_{p|\Pi}1/p}$), and the average is actually dominated by a small number of choices of $b$ which have many prime factors. This is analogous to the well-known fact from the anatomy of integers that a \textit{typical} integer of size $y$ (i.e. for all elements of a subset of $\{1,\dots,y\}$ of size $y-o(y)$) has roughly $\log\log{y}$ prime factors and so about $2^{\log\log{y}}=(\log{y})^{\log{2}}$ divisors, but the \textit{average} number of divisors of integers of size $y$ is roughly $\log{y}$, which is much larger (because the average is dominated by the contribution of a set of size $o(y)$ which have unusually many divisors).

The outcome of this observation is that a \textit{typical} choice of $b/\Pi$ actually has \textit{no} representations as $(a+1)/q$ with $q|\Pi$ and $(a,q)=1$, since it typically only has a moderate number of representatitves as  $(a+1)/q$, and it is unlikely that any of these satisfy $(a,q)=1$. Once we have restricted to those $b$ which do have a representation, they actually have \textit{many} representations on average. Since we have found a set of rationals which typically have many representations, we can now derive the desired contradiction in a similar manner to before.

The case of more general shifts $\gamma$ is a natural generalization of the case when $\gamma=1$. When $\gamma=r/s$ essentially the same argument applies directly, with mild dependencies on the height $\max(r,s)$ of $\gamma$. 

\section{Proof of the Theorem}

\subsection*{Notation}
We use the Vinogradov notations $\ll,\gg$, as well as denote with $\varphi$ the Euler totient function, and with $\omega$ the number of distinct prime divisors. Sums and products over $p$ are understood as sums respectively products over prime numbers.\\

We use a standard key auxiliary statement from which the proof of the Theorem then concludes immediately.

\begin{lem}\label{obv_key_lem}
We define 
\[A_q(\psi,\gamma) := 
\left\{\alpha \in [0,1): \exists a \in \mathbb{Z}: \gcd(a,q) = 1, \left\lvert \alpha - \frac{a + \gamma}{q}\right\rvert \leq \frac{\psi(q)}{q}\right\}\]
and assume $\gamma \neq 0$ to satisfy \eqref{dioph_cond}.
Then for every $\varepsilon > 0$ and all $X$ sufficiently large, there exists $Y>X$ and $\psi_X: \mathbb{N} \to [0,\infty)$ with support in $[X,Y]$ such that 
\[\sum_{q \in [X,Y]} \frac{\psi_X(q)\varphi(q)}{q} \geq \frac{1}{3},\]
but 
\[\lambda\left(\bigcup_{q \in [X,Y]}A_q(\psi_X,\gamma)\right) < \varepsilon.\]
\end{lem}

\begin{proof}[Proof of Theorem \ref{thm_counter} assuming Lemma \ref{obv_key_lem}]
    This is a routine argument. By an application of Lemma \ref{obv_key_lem}, we obtain 
    sequences $(X_k)_k,(Y_k)_k,(\psi_{X_k})_k$ such that:

    \begin{itemize}
        \item[(i)] $X_1 < Y_1 < X_2 < Y_2 \ldots$,
        \item[(ii)] \[\sum_{q \in [X_k,Y_k]} \frac{\psi_{X_k}(q)\varphi(q)}{q} \geq \frac{1}{3},\]
       \item[(iii)] 
\[\lambda\left(\bigcup_{q \in [X_k,Y_k]}A_q(\psi_{X_k},\gamma)\right) < 2^{-k}.\]
    \end{itemize}
    Next, we define $\psi = \sum_{k \in \mathbb{N}} \psi_{X_k}$ and observe that by (i), for every $q \in \mathbb{N}$, we have $\psi(q) = \psi_{X_k}(q)$ for some $k$. By (ii), we have
    \[\sum_{q \in \mathbb{N}} \frac{\psi(q)\varphi(q)}{q} = \infty,\]
    so it suffices to show that
    $\lambda(W'(\psi,\gamma)) = 0$. Note that
    \[W'(\psi,\gamma) = \limsup_{q \to \infty } A_q(\psi,\gamma)
    \subseteq \bigcap_{k \in \mathbb{N}} \bigcup_{\ell \geq k} \bigcup_{q \in [X_{\ell},Y_{\ell}]}A_q(\psi_{X_{\ell}},\gamma).
    \]
    Thus by (iii),
    \[\lambda(W'(\psi,\gamma)) \leq \lim_{k \to \infty} \sum_{\ell \geq k} \lambda\left(\bigcup_{q \in [X_{\ell},Y_{\ell}]}A_q(\psi_{X_{\ell}},\gamma)\right) \leq \lim_{k \to \infty} \sum_{\ell \geq k} 2^{-{\ell}} =  0.\]
\end{proof}

It remains to provide a proof for Lemma \ref{obv_key_lem}.
We first provide a particularly transparent proof for the special case of $\gamma = 1$. We will then provide, with the same mechanism underlying, a slightly more technical argument to extend the counterexamples to all $\gamma$ that satisfy the Diophantine condition \eqref{dioph_cond}.\\

\subsection*{The proof for $\gamma = 1$}
We fix $\varepsilon > 0$ and let $X = X(\varepsilon)$ be chosen later. We fix consecutive primes $X < p_1<\dots <p_k$ and set 
\[
\Pi_X := \prod_{i = 1}^k p_i\qquad \text{and}\qquad\Sigma_X := \sum_{i = 1}^k \frac{1}{p_i}.
\]
As $k\rightarrow\infty$ we have $\Sigma_X\rightarrow\infty$, so by choosing $k$ sufficiently large in terms of $X$ and $\varepsilon$ we may assume that also $\Sigma_X$ is sufficiently large in terms of $\varepsilon$ and $X$.
We then define
\begin{align*}
\psi_X(q) &:= \begin{cases}
     \frac{q}{2\Pi_X} &\text{ if } q \mid \Pi_X, \,q > 1,\\
    0 &\text{ otherwise,}
\end{cases}\\
Y &:= \Pi_X.\end{align*}
We see that $\psi_X$ is supported on $[X,Y]$. Moreover, we have that 
\[\sum_{q \in [X,Y]} \frac{\psi_X(q)\varphi(q)}{q} 
= \frac{1}{2\Pi_X}\sum_{q \mid \Pi_X}\varphi(q) - \frac{1}{2\Pi_X} = \frac{1}{2}-\frac{1}{2\Pi_X}\geq \frac{1}{3}.
\]
Thus we are left to establish a suitable upper bound for $\lambda(\bigcup_q A_q(\psi_X,1))$. For readability, we now drop the subscripts $X$ from $\Pi_X,\Sigma_X$ and $\psi_X$. We note that for any $q \in \supp \psi$, we have $\frac{\psi(q)}{q} = \frac{1}{2\Pi}$,
and consequently, 
\[A_q(\psi,1) = \bigcup_{\gcd(a,q) = 1}E_{a,q},\]
where
\[
E_{a,q}:=\Bigl[\frac{a+1}{q}-\frac{1}{2\Pi},\frac{a+1}{q}+\frac{1}{2\Pi}\Bigr].
\]

Since $q\mid \Pi$, all $E_{a,q}$ are intervals around rationals of the form $\frac{b}{\Pi}, b \in \mathbb{Z}$. Thus

\begin{align*}
\lambda\Bigl(\bigcup_{q \in [X,Y]}A_q(\psi,1)\Bigr)
&\subseteq 
\lambda\Bigl(\bigcup_{q\mid \Pi}\bigcup_{\substack{\gcd(a,q)=1}}E_{a,q}\Bigr)\\
&=\lambda\Bigl(\bigcup_{\substack{0 \leq b < \Pi\\ \frac{b}{\Pi}=\frac{a+1}{q}\text{ for some $\gcd(a,q)=1$}}}\Bigl[\frac{b}{\Pi}-\frac{1}{2\Pi},\frac{b}{\Pi}+\frac{1}{2\Pi}\Bigr]\Bigr)\\
&=\frac{1}{\Pi}\sum_{\substack{0 \leq b < \Pi\\ \frac{b}{\Pi}=\frac{a+1}{q}\text{ for some $\gcd(a,q)=1$}}}1.
\end{align*}
We split the sum into two subsums according to whether $\sum_{p \mid \gcd(\Pi,b)}1 < 1.01\Sigma$ or not: We define
\begin{align*}
S_1&:=\sum_{\substack{0 \leq b < \Pi\\ \frac{b}{\Pi}=\frac{a+1}{q}\text{ for some $\gcd(a,q)=1$}\\ \sum_{p \mid \gcd(\Pi,b)}1 < 1.01\Sigma}}1,\\
S_2&:=\sum_{\substack{0 \leq b < \Pi\\ \frac{b}{\Pi}=\frac{a+1}{q}\text{ for some $\gcd(a,q)=1$}\\ \sum_{p \mid \gcd(\Pi,b)}1 \geq 1.01\Sigma}}1.
\end{align*}
This gives
\[\lambda\left(\bigcup_{q \in \supp\psi}A_q(\psi,1)\right)\leq \frac{1}{\Pi} (S_1+S_2),\]
and so it suffices to show $\max\{S_1,S_2\} \leq \Pi\varepsilon/2$.\\

For $S_2$, we drop the sieve condition, and bound it using the fact that it is rare for $b$ to have so many prime factors.
For $\alpha>0$ we have
\[
\mathds{1}_{\left\{\sum_{p \mid \gcd(\Pi,b)}1 \geq 1.01\Sigma\right\}}
\leq \exp\Bigl(\alpha\Bigl(\sum_{p \mid \gcd(\Pi,b)}1 - 1.01\Sigma\Bigr)\Bigr).
\]
Therefore, applying this bound and rearranging the sum according to $d=\gcd(\Pi,b)$, we obtain
\begin{align*}S_2 &\leq \sum_{\substack{0 \leq b < \Pi}}\exp\Bigl(\alpha\Bigl(\sum_{p \mid \gcd(\Pi,b)}1 - 1.01\Sigma\Bigr)\Bigr)\\
&= \sum_{\substack{d \mid \Pi}}
\exp\Bigl(\alpha\Bigl(\sum_{p \mid d}1 - 1.01\Sigma\Bigr)\Bigr)
\sum_{\substack{0 \leq b < \Pi\\\gcd(\Pi,b) = d}}1\\
&= \varphi(\Pi)\exp\Bigl( - 1.01\alpha\Sigma\Bigr)\sum_{\substack{d \mid \Pi}}
 \frac{\exp\Bigl(\alpha\sum_{p \mid d}1\Bigr)}{\varphi(d)}.
\end{align*}
We note that $\varphi(\Pi)\ll \Pi \exp(-\Sigma)$ and that $1/\varphi(d)\ll \prod_{p|d}(1+1/p)/p$. Thus we obtain the bound
\begin{align*}S_2 
&\ll \Pi \exp((-1 -1.01\alpha)\Sigma)
\sum_{d \mid \Pi} \prod_{p \mid d}\frac{e^{\alpha}\left(1 + \frac{1}{p}\right)}{p}\\
&= \Pi \exp((-1 -1.01\alpha)\Sigma)
\prod_{p \mid \Pi} \left(1 + \frac{e^{\alpha}(1+1/p)}{p}\right)\\
&\ll \Pi \exp((e^{\alpha}-1 -1.01\alpha)\Sigma).
\end{align*}
For $\alpha = \frac{1}{100}$, we have 
$e^{\alpha}-1 -1.01\alpha <- 10^{-6} < 0$, thus 
$S_2 \ll \Pi \exp(-10^{-6}\Sigma)$.\\

Thus we are left to provide upper bounds for $S_1$. We note that the condition `$\frac{b}{\Pi}=\frac{a+1}{q}$ for some $\gcd(a,q)=1$' is equivalent to (setting $e = \frac{\Pi}{q}$)
\[
\exists e \mid \gcd(b,\Pi):\, \gcd\Bigl(\tfrac{b}{e}-1,\tfrac{\Pi}{e}\Bigr)=1.
\]
By the union bound on $e$, we have
\[
S_1=\sum_{\substack{0 \leq b < \Pi \\ \ \exists e \mid \gcd(b,\Pi):\, \gcd\left(\tfrac{b}{e}-1,\tfrac{\Pi}{e}\right)=1\\ \sum_{p \mid \gcd(b,\Pi)}1 < 1.01\Sigma}}1\le \sum_{\substack{0 \leq b < \Pi\\ \sum_{p \mid \gcd(b,\Pi)}1 < 1.01\Sigma}}\sum_{\substack{e \mid \gcd(b,\Pi)\\ \gcd\left(\frac{b}{e}-1,\frac{\Pi}{e}\right)=1}}1.
\]
We let $d=\gcd(b,\Pi)$ and $b'=b/d$. Rearranging the sum then gives
\begin{align*}
S_1 &\leq  \sum_{\substack{d \mid \Pi\\ \omega(d) < 1.01\Sigma}} \sum_{e \mid d}\sum_{\substack{0 \leq b' < \frac{\Pi}{d}\\ \gcd\left(b'\frac{d}{e}-1,\frac{\Pi}{e}\right)=1\\ \gcd\left(b',\frac{\Pi}{d}\right)=1}}1\\
&\leq \sum_{\substack{d \mid \Pi\\ \omega(d) < 1.01\Sigma}} \sum_{e \mid d}\sum_{\substack{0 \leq b' < \frac{\Pi}{d}\\ \gcd\left(b'\frac{d}{e}-1,\frac{\Pi}{d}\right)=1\\ \gcd\left(b',\frac{\Pi}{d}\right)=1}}1.
\end{align*}
Recalling that $\Pi$ is square-free, for every prime $p \mid \frac{\Pi}{d}$, we have $p \nmid d$, thus $p \nmid \frac{d}{e}$. Consequently, the sieve conditions from the inner sum correspond to avoiding the \textit{distinct} residue classes $0$ and $\left(\frac{d}{e}\right)^{-1} \pmod{p}$ for all $p \mid \frac{\Pi}{d}$.
We remark that this is the crucial part where the proof for $\gamma = 0$ would break down, since there, the two sieve conditions correspond to the same residue class.
Since the sum over $b'$ runs through a complete set of residue classes modulo $\frac{\Pi}{d}$, the Chinese Remainder Theorem implies 

\[\sum_{\substack{0 \leq b' < \frac{\Pi}{d}\\ \gcd\left(b'\frac{d}{e}-1,\frac{\Pi}{d}\right)=1\\ \gcd\left(b',\frac{\Pi}{d}\right)=1}}1 = \prod_{p \mid \frac{\Pi}{d}} (p-2)
\ll \frac{\Pi}{d} \exp(-2 \Sigma) \frac{d^2}{\varphi(d)^2}.
\]
Further, since 
$\omega(d) < 1.01 \Sigma$, we have

\[\sum_{e \mid d}1 \leq 2^{1.01 \Sigma} = \exp(1.01 \cdot \log 2 \cdot \Sigma).\]
Putting in these estimates in the above, we obtain

\[\begin{split}S_1 &\ll \Pi \exp((1.01 \cdot \log 2 -2) \cdot \Sigma)
\sum_{d \mid \Pi} \frac{d}{\varphi(d)^2}
\\&= \Pi \exp((1.01 \cdot \log 2 -2) \cdot \Sigma)\prod_{p \mid \Pi}\left(1 + \frac{p}{(p-1)^2}\right) \\&\ll \Pi \exp((1.01 \cdot \log 2 -1) \cdot \Sigma).
\end{split}\]
Since $1.01 \cdot \log 2 -1 <  - 10^{-6}$, we obtain 

\[\max\{S_1,S_2\} \ll \Pi \exp(-10^{-6} \Sigma).\]

We recall that since $k$ is chosen sufficiently large in terms of $X$ and $\varepsilon$, we can make $\Sigma$ arbitrarily large in terms of $\varepsilon$. Thus, for a suitable choice of $k$ this shows $\max\{S_1,S_2\} \leq \frac{\Pi\varepsilon}{2}$, as desired. This concludes the proof of Lemma \ref{obv_key_lem} for $\gamma = 1$.\\

\subsection*{The general case}
% In order to obtain Lemma \ref{obv_key_lem} for all $\gamma \neq 0$ satisfying 
% \eqref{dioph_cond}, we follow the above, but make some necessary changes: 
% We replace the value of $1.01$ by a variable $c > 1$. Following the lines above, we get
% \[\max\{S_1,S_2\} \ll 
% \Pi\left(\exp\left(\max\{e^{\alpha} - 1 - c\alpha,c \log 2 - 1\}\right)\Sigma\right).
% \]
% Minimizing the right-hand side in $c,\alpha$ yields 
% $c = \frac{e}{2}, \alpha = 1 - \log 2$, in which case we obtain
% \[\max\{S_1,S_2\} \ll 
% \Pi \exp\left((\frac{e}{2}\log 2 - 1)\Sigma\right) \ll \Pi \exp(- 0.05 \Sigma).\]

We may assume the sequence of $\frac{r_n}{s_n}$ with \eqref{dioph_cond} to be as sparse as we may. We remark that we do not assume $\frac{r_n}{s_n}$ to be reduced, thus this case covers both irrationals and rationals.
Let us fix $\varepsilon$ and $X$ as before,  and fix constants $\rho > \frac{1}{12}$ and $c>1$ (to be optimised later). We write $\frac{r}{s} = \frac{r_n}{s_n} \neq 0$ with $n$ chosen sufficiently large in terms of $X,\varepsilon,\rho,c$, and let $X < p_1 < p_2 < \ldots$ denote the primes not dividing $rs$. We then choose $k$ minimal such that
\[s \leq \exp\left(\rho\sum_{i = 1}^k \frac{1}{p_i}\right).\]
Writing $\Sigma := \sum_{i = 1}^k \frac{1}{p_i}$, minimality gives, for $n$ sufficiently large,
\begin{equation}\label{asymp_s}
\frac{1}{2}\exp\left(\rho\Sigma\right) \leq s \leq \exp\left(\rho\Sigma\right).
\end{equation}
Again, let $\Pi := \prod_{i = 1}^k p_i$ and $Y := \Pi$. We define
\[\psi(q) = \begin{cases}
    \frac{q}{2\Pi} &\text{ if } q \mid \Pi, \,\omega\left(\frac{\Pi}{q}\right) < c\Sigma,\, q > 1,\\
    0 &\text{ otherwise}.
\end{cases}\]
Clearly, this shows that we have $\supp\psi\subseteq[X,Y]$.
We observe
\[\sum_{q \in \mathbb{N}} \frac{\varphi(q)\psi(q)}{q}
= \frac{1}{2\Pi}\sum_{q \mid \Pi} \varphi(q) - 
\frac{1}{2\Pi}\sum_{\substack{q \mid \Pi\\
\omega\left(\frac{\Pi}{q}\right) \geq c\Sigma
}} \varphi(q).
\]
By using the same bound that was used on $S_2$ in the case $\gamma=1$, we obtain for any choice of $\alpha>0$

\[\frac{1}{\Pi}\sum_{\substack{q \mid \Pi\\
\omega\left(\frac{\Pi}{q}\right) \geq c\Sigma
}} \varphi(q) \leq \exp( (e^\alpha-1-c\alpha)\Sigma).\]
Choosing $\alpha=\log{c}$ then shows that, provided $s$ is large enough in terms of the fixed constants $\rho,c$, this is less than $1/10$. Since $s=s_n$, this immediately implies 
that for a sufficiently large choice of $n$ we have
\[
\sum_{q \in \mathbb{N}} \frac{\varphi(q)\psi(q)}{q} \geq \frac{1}{3}.
\]
This gives the first claim of Lemma \ref{obv_key_lem}.

We now proceed to bound $\lambda(\bigcup_qA_q(\psi,\gamma))$. First we show that \eqref{dioph_cond} allows us to replace $\gamma$ with $\frac{r}{s}$ at the cost of an acceptable error term. Note that $X$ is fixed and we have $\lvert r\rvert \asymp s$, so for $n$ sufficiently large in terms of $X$ we have
\[\sum_{p \mid rs} \frac{1}{p} \leq \omega(\lvert rs\rvert) \ll \frac{\log s}{\log \log s} = o(\Sigma),\]
where here (and throughout this proof) $o(1)$ is understood as $n\rightarrow\infty$ for fixed $c,\rho,X,\varepsilon$. Hence, by Mertens' estimate,
\[\Sigma = (1+o(1))\log \log k+O(\log\log{X}).\]
Thus, using \eqref{asymp_s} and $k$ being sufficiently large in terms of $X$, we see that
\[k = \exp(s^{\rho^{-1}+o(1)}).\]
Recall that $q \in \supp(\psi)$ implies $\omega\left(\frac{\Pi}{q}\right) < c\Sigma$, and thus
$\frac{\Pi}{q} \leq p_k^{c \Sigma}$.
Hence for $q \in \supp(\psi)$,

\[\psi(q) 
\geq \frac{1}{2}\exp\left(-c\Sigma \log p_k\right)
= \exp\left(-s^{\rho^{-1}+o(1)}\right)
\geq \exp\left(-s^{12}\right),
\]
provided $n$ is sufficiently large. By \eqref{dioph_cond},
\[\left\lvert\frac{r}{s} - \gamma \right\rvert < \exp\left(-s^{12}\right) \leq \psi(q).\]
This allows us to write
\[E_{a,q}(\gamma) \subseteq E_{a,q}^{(2)}\left(\frac{r}{s}\right),\]
where $E_{a,q}^{(2)}\left(\gamma\right)$ denotes the interval around $\frac{a + \gamma}{q}$ with radius $\frac{2\psi(q)}{q}$. Consequently, we may upper-bound
\[\bigcup_{q \in \supp \psi} A_q(\psi,\gamma)
\subseteq \bigcup_{q \mid \Pi} A_q\left(2\psi,\frac{r}{s}\right).
\]
We see that $A_q\left(2\psi,\frac{r}{s}\right)$ has its centers at
$\frac{a + \frac{r}{s}}{q} = \frac{sa+r}{sq}$, thus are in particular of the form
$\frac{b}{s\Pi}$.
Again, we split into two subsums according to whether $\sum_{p \mid \gcd(s\Pi,b)}1 < c\Sigma$ or not: We define

\begin{align*}
S_1&:=\sum_{\substack{0 \leq b < s\Pi\\ \frac{b}{s\Pi}=\frac{as+r}{sq}\ \text{for some $\gcd(a,q)=1$}\\ \sum_{p \mid \gcd(b,s\Pi)}1 < c\Sigma}}1,\\
S_2&:=\sum_{\substack{0 \leq b < s\Pi\\ \frac{b}{s\Pi}=\frac{as+r}{sq}\ \text{for some $\gcd(a,q)=1$}\\ \sum_{p \mid \gcd(b,s\Pi)}1 \geq c\Sigma}}1,\\
\end{align*}
and obtain
\[
\lambda\left(\bigcup_{q \in [X,Y]}A_q(\psi,\gamma)\right)\ll \frac{1}{\Pi}(S_1+S_2).
\]
We recall that
$\omega(s) = o(\Sigma)$, thus the contribution of primes $p \mid s$ in the above conditions is negligible.\\

For $S_2$, we drop both the sieve condition $\gcd(a,q)=1$ again as well as the congruence condition. We thus obtain, following the proof for $\gamma =1$, that for any choice of $\alpha>0$

\[\begin{split}S_2 &\ll \sum_{\substack{0 \leq b < s\Pi \\ \sum_{p \mid \gcd(s\Pi,b)}1 \geq c\Sigma}}\exp\Bigl(\alpha\Bigl(\sum_{p \mid \gcd(s\Pi,b)}1 - c\Sigma\Bigr)\Bigr) \ll s\Pi \exp((e^{\alpha}-1 -c\alpha + o(1))\Sigma).\end{split}\]

We are left to provide upper bounds for $S_1$, again following the proof strategy for the case $\gamma = 1$. We note that (writing $e=\frac{\Pi}{q}$) the condition `$\frac{b}{s\Pi}=\frac{sa+r}{sq}$ for some $\gcd(a,q)=1$' is equivalent to
\[
\exists e \mid \gcd(b,\Pi):\, \gcd\Bigl(\frac{b}{e}-r,\frac{\Pi}{e}\Bigr)=1, \quad \frac{b}{e} \equiv r \pmod s.
\]
Thus we see that by the union bound on $e$,
\begin{align*}
S_1&\leq\sum_{\substack{0 \leq b < s\Pi\\ \ \exists e \mid \gcd(b,\Pi): \gcd\left(\frac{b}{e}-r,\frac{\Pi}{e}\right)=1, \ \frac{b}{e}\equiv r \pmod{s}\\ \sum_{p \mid \gcd(b,\Pi)}1 < c\Sigma}}1\\
&\le 
\sum_{\substack{d \mid \Pi\\ \omega(d) < c\Sigma}}
\sum_{e \mid d} \sum_{\substack{0 \leq b' < \frac{s\Pi}{d}\\ \gcd\left(b'\frac{d}{e}-r,\frac{\Pi}{e}\right)=1\\ \gcd\left(b',\frac{\Pi}{d}\right)=1\\ {\frac{b'd}{e}\equiv r \pmod{s}}}}1
\\&\le 
\sum_{\substack{d \mid \Pi\\ \omega(d) < c\Sigma}}
\sum_{e \mid d} \sum_{\substack{0 \leq b' < \frac{s\Pi}{d}\\ \gcd\left(b'\frac{d}{e}-r,\frac{\Pi}{d}\right)=1\\ \gcd\left(b',\frac{\Pi}{d}\right)=1\\ {\frac{b'd}{e}\equiv r \pmod{s}}}}1.
\end{align*}
Since $rs$ is coprime to $\Pi$, the Chinese Remainder Theorem implies as before
\[\sum_{\substack{0 \leq b' < \frac{s\Pi}{d}\\ \gcd\left(b'\frac{d}{e}-r,\frac{\Pi}{d}\right)=1\\ \gcd\left(b',\frac{\Pi}{d}\right)=1\\ \frac{b'd}{e}\equiv r \pmod{s}}}1
= \prod_{p \mid \frac{\Pi}{d}}(p-2)
\ll \frac{\Pi}{d} \exp(-2\Sigma) \frac{d^2}{\varphi(d)^2},\]
and concluding as in the case of $\gamma = 1$, we obtain

\[S_1 \ll  \Pi \exp((c \cdot \log 2 -1) \cdot \Sigma),\]
which shows
\[\begin{split}\max\{S_1,S_2\} &\ll  \Pi \max\{s \exp((e^{\alpha}-1 -c\alpha+o(1))\Sigma), \exp((c \cdot \log 2 -1)\Sigma)\}
\\&\ll \Pi\max\{\exp((e^{\alpha}-1-c\alpha+\rho+o(1))\Sigma), \exp((c \cdot \log 2-1)\Sigma)\},\end{split}\]
where we used in the last line $s \ll \exp(\rho \Sigma)$.\\

We now want to choose $c>1,\alpha>0,\rho>1/12$ such that $e^\alpha-1-c\alpha+\rho<0$ and $c\log{2}-1<0$. An elementary optimization shows that the supremal admissible value of $\rho$ is 
\[
\frac{-1+\log{2}-\log{\log{2}}}{\log{2}}=0.086\ldots>\frac{1}{12},
\]
occurring when $c\approx 1/\log{2}$ and $\alpha\approx -\log{\log{2}}$. Hence we may fix $\rho > \frac{1}{12}$, $c>1$ and $\alpha>0$ such that $e^{\alpha}-1-c\alpha+\rho<0$ and $c\log 2-1<0$. Absorbing the $o(1)$ term, there exists some fixed $\eta>0$ for which
\[
\max\{S_1,S_2\}\ll\Pi\exp(-\eta\Sigma)
\]
whenever $n$ is large enough in terms of $X$. Choosing $n$ such that $\Sigma$ is sufficiently large in terms of $\varepsilon$ concludes the proof.

\section{AI declaration}

This work was done without the use of artificial intelligence tools; it is solely the work of the authors.
    \bibliographystyle{plain}

\bibliography{bibliography.bib}
\Addresses

\end{document}